\documentclass[a4paper,reqno]{amsart}
\usepackage{a4wide,amsfonts,amssymb,mathptmx}
\usepackage{hyperref}

\newtheorem{thm}{Theorem}
\numberwithin{thm}{section}
\newtheorem{prop}[thm]{Proposition}
\newtheorem{lem}[thm]{Lemma}

\theoremstyle{definition}

\newtheorem{exmp}[thm]{Example}

\theoremstyle{remark}

\newcommand{\ip}[2]{\ensuremath{({#1}\,|\,{#2})}} 
\providecommand{\BBb}[1]{{\mathbb{#1}}}

\newcommand{\Bcirc}{\overset{\lower 1.5pt%
              \hbox{$@,@,@,@,@,\scriptscriptstyle\circ$}}B{}}
\newcommand{\Binfty}{\overset{\lower 1.5pt%
              \hbox{$@,@,@,@,@,\scriptscriptstyle\infty$}}B{}}
\newcommand{\bigdot}{\mathbin{\raise.65\jot\hbox{$\scriptscriptstyle\bullet$}}}

\newcommand{\B}{{\BBb B}}
\newcommand{\C}{{\BBb C}}

\newcommand{\erd}{\overset{\lower 1pt\hbox{\large.}}{e}
                  \overset{\lower 1pt\hbox{\large.}}{r}}

\newcommand{\Fcirc}{\overset{\lower 1.5pt%
               \hbox{$@,@,@,@,@,\scriptscriptstyle\circ$}}F{}}
\newcommand{\fracc}[2]{{
                \textstyle\frac{#1}{\raise 1pt\hbox{$\scriptstyle #2$}}}}

\newcommand{\fracci}[2]{{\frac{#1}{\raise 1pt\hbox{$\scriptscriptstyle #2$}}}}

\newcommand{\mlap}{-\!\operatorname{\Delta}}

\newcommand{\op}[1]{\operatorname{#1}}

\renewcommand{\Re}{\operatorname{Re}}

\newcommand{\R}{{\BBb R}}
\newcommand{\Rn}{{\BBb R}^{n}}

{
\end{minipage}%
\par\medskip\noindent}%
\newcounter{enmcount}\renewcommand{\theenmcount}{{\rm\arabic{enmcount}}}

\newcounter{rmcount}\renewcommand{\thermcount}{{\rm\roman{rmcount}}}

\newcounter{Rmcount}\renewcommand{\theRmcount}{{\rm\Roman{Rmcount}}}
\newenvironment{Rmlist}{%
\begin{list}{{\rm(\theRmcount)}}{\setlength{\labelwidth}{\leftmargin}%
\usecounter{Rmcount}}}{\end{list}}

\newcommand{\set}[2]{\{\,#1 \mid #2\,\}}
\newcommand{\Set}[2]{\bigl\{\,#1\bigm| #2\,\bigr\}}

\begin{document}
\title[Log-Convex Decay in Non-Selfadjoint Dynamics]{On a Criterion for  Log-Convex Decay in Non-Selfadjoint Dynamics}
\author{Jon Johnsen}
\address{Department of Mathematical Sciences, Aarhus University, Ny Munkegade 118, Building 1530, 8000 Aarhus, Denmark}
\email{jojoh@math.au.dk}
\begin{abstract} 
The short-time and global behaviour are studied for 
autonomous linear evolution equations defined by generators
of uniformly bounded holomorphic semigroups in a Hilbert space.
A general criterion for log-convexity in time of the norm of the solution is treated. 
Strict decrease and differentiability at the initial time  results,  with a derivative
controlled by the lower bound of the negative generator, which is proved
strictly accretive with equal numerical and spectral abscissas.
\end{abstract}
\maketitle

\section{Introduction}
The subjects here are the \emph{global} and the \emph{short-time} behaviour of the solutions to 
the Cauchy problem of an autonomous linear evolution equation, throughout with data $u_0\ne0$,
\begin{equation}  \label{ivp-id}
   \partial_t u+Au=0 \quad\text{ for $t>0$},\qquad u(0)=u_0\quad\text{ in $H$}.
\end{equation}
In case the generator $-A$ is non-selfadjoint, this is particularly interesting.
``Non-self-adjoint operators is an old, sophisticated and highly developed subject''
to quote the recent treatise of Sj{\"o}strand \cite{Sj19};  
also the exposition of Helffer \cite[Ch.\ 13]{Hel13} on their pseudo-spectral theory could be mentioned; or \cite{EmTr05}. 

Logarithmically convex decay of 
the solutions was seemingly first studied in the author's paper \cite{JJ18logconv}. This is given a more concise exposition here, with additional examples.

The main purpose below, however, is to improve the results in \cite{JJ18logconv} by adding in Section~2
a much sharper necessary condition on $A$ for the log-convex decay, leading to the improved Theorem~7 below.

\bigskip

It is assumed that $A$ is an accretive operator with domain $D(A)$ in a complex Hilbert space $H$,
with norm $|\cdot|$ and inner product $\ip{\cdot}{\cdot}$, and that $-A$ generates a uniformly bounded, holomorphic
$C_0$-semigroup $e^{-zA}$ for $z$ in an open sector having the form
$\Sigma_{\delta}=\set{z\in\C}{-\delta< \arg  z<\delta}$.  Focus is here on the
``height'' function 
\begin{equation}
  h(t)=|e^{-tA}u_0|.
\end{equation}
This was shown in \cite{JJ18logconv} to be a
\emph{log-convex} function, that is, for $0\le r\le s\le t<\infty$  
\begin{equation} \label{h-logcon}
  \big|e^{-sA}u_0\big| \le \big|e^{-rA}u_0\big|^{1-\frac{s-r}{t-r}} \big |e^{-tA}u_0\big|^{\frac{s-r}{t-r}},
\end{equation}
\emph{if and only if} the possibly non-normal generator $-A$ has the special property
that for every $x\in D(A^2)$, 
\begin{equation} \label{A-cond}
   2\big(\Re \ip{Ax}{x}\big)^2 \le\Re\ip{A^2x}{x}|x|^2+|Ax|^2 |x|^2.
\end{equation}
The present paper and \cite{JJ18logconv} grew out of the author's joint work \cite{ChJo18,ChJo18ax} 
on the inverse heat equation and its well-posedness under the Dirichlet condition. But the main parts also apply to solutions of the 
similar Neumann problem studied in \cite{JJ19,JJ19cor,JJ19iso}.

To elucidate the importance of  \eqref{h-logcon}, hence of \eqref{A-cond}, two remarks are made.

$1^\circ$ The log-convexity in \eqref{h-logcon} implies that the solutions $u$ of \eqref{ivp-id} have important
 global properties  in common with those of the heat equation (the case $A=\mlap$ in $H=
 L_2(\Omega) $ for a bounded domain $\Omega\subset \Rn$). Namely, the height function
 $h(t)=|e^{tA}u_0|$ of \eqref{ivp-id} is 
\begin{quote}
  (i)\hphantom{ii}\quad strictly positive ($h>0$),
\par\noindent
  (ii)\hphantom{i}\quad strictly decreasing ($h'<0$),
 \par\noindent 
  (iii)\quad strictly convex ($\Leftarrow\; h''>0$).
\end{quote}
Here the strict decrease and strict convexity combine to a noteworthy and precise dynamical
property. For example, even if $A$ has eigenvalues in
$\C\setminus\R$, they do not give rise to oscillations in the \emph{size} of the
solution $e^{-tA}u_0$---this is ruled out by strict convexity, which
thus can be seen as a stiffness in the decay of $h(t)$.

In addition,  \eqref{ivp-id} also shares the short-time behaviour with the heat equation, for
in terms of the numerical range $\nu(A)=\Set{\ip{Ax}{x}}{x\in D(A),\ |x|=1}$ and its lower bound $m(A)=\inf \,\Re\nu(A)$,
the onset of decay of $h\in C^\infty(\,]0,\infty[\,)\cap C([0,\infty[\,)$  is constrained by the properties:
\begin{quote}
    (iv)\quad $h(t)$ is right differentiable at $t=0$, with
\par\noindent     
    (v)\hphantom{i}\quad $h'(0)\le -m(A)<0$  for $|u_0|=1$, though 
\par\noindent  
    (vi)\quad $h'(0)=-\Re\ip{Au_0}{u_0}$ whenever $u_0\in D(A)$, $|u_0|=1$.
 \end{quote}
For the considered $A$, (iv)--(vi) follow from log-convexity; cf.\ the below Theorem~\ref{logcon-thm}.

More generally, one could
try to work with the $A$ that merely have strictly convex height functions, but this class is not
easy to characterise. One may therefore view \eqref{A-cond} as a very large class of (possibly
non-normal) generators having the described dynamical properties in common with the selfadjoint
cases.

$2^\circ$ Secondly, the operators satisfying \eqref{A-cond} may be seen to comprise the $A$ that are selfadjoint, $A^*=A$, or 
normal, $A^*A=AA^*$.  But as observed in \cite{ChJo18ax}, one only needs the following two half-way houses,
\begin{equation} \label{hypo-normal}
   D(A)\subset D(A^*),\qquad |Ax|\ge |A^*x| \text{ for every $x\in D(A)$}.
\end{equation}
This property is \emph{hyponormality} for unbounded operators, as studied by Janas~\cite{Jan94}. 
Clearly $A$ is normal if and only if both $A$, $A^* $ are hyponormal, so
this operator class is quite general. As symmetric operators have a full inclusion
$A\subset A^*$, they are also encompassed by the hyponormal class. But there is more:

\begin{exmp}  
  Truly hyponormal operators are easily exemplified: for the advection-diffusion operators
  $A^{\pm}u=-u''\pm u'$ in $L^2(\alpha,\beta)$, for $\alpha<\beta$ in $\R$, it is classical that the
  minimal realisation $A^{\pm}_{\min}$ has the domain $D(A^{\pm}_{\min})=H^2_0(\alpha,\beta)$ because of the
  ellipticity (cf.\ \cite[Thm.\ 6.24]{G09}). The maximal realisation has domain
  $D(A^{\pm}_{\max})=H^2(\alpha,\beta)$, for when $f=-u''\pm u'$ holds for $u$, $f\in L^2$, then
  $-u'\pm u\in L^2$ as primitives of $f$, so $u'\in L^2$; hence $u''\in L^2$. 
  Via the formal adjoints $A^{\mp}$ this gives (cf.\ \cite[Lem.\ 4.3]{G09})
  \begin{equation}
    D(A^{\pm}_{\min})\subsetneq D(A^{\mp}_{\max})= D((A^{\pm}_{\min})^*).
  \end{equation}
  Partial integration for $u\in H^2(\alpha,\beta)$ yields
  $\|-u''\pm u'\|^2 =\| u''\|^2 + \|u'\|^2 \mp (|u'(\beta)|^2-|u'(\alpha)|^2)$, where the last two
  terms vanish for $u\in D(A^{\pm}_{\min})$, so that $\| A^{\pm}_{\min}u\|= \| (A^{\pm}_{\min})^*
  u\|$. Hence the $A^{\pm}_{\min}$ are nonnormal, but nonetheless hyponormal.
\end{exmp}

That every  hyponormal operator $A$ in $H$ necessarily satisfies the log-convexity condition
\eqref{A-cond} is recalled from \cite{JJ18logconv} for the reader's convenience: the inclusion
$D(A)\subset D(A^*)$ gives at once for $x\in D(A^2)$ that
\begin{equation} \label{hypo-ver}
  2(\Re\ip{Ax}{x})^2 \le \frac12|(A+A^*)x|^2|x|^2 
  \le \big(|Ax|^2+\Re\ip{A^2x}{x}\big)|x|^2,
\end{equation}
for in the last step the norm inequality in \eqref{hypo-normal} gives,
because $D(A^2)\subset D(A)\subset D(A^*)$, that
\begin{equation}
  |(A+A^*)x|^2= |Ax|^2+|A^*x|^2 +2\Re\ip{Ax}{A^*x}
  \le 2|Ax|^2+2\Re\ip{A^2x}{x}.
\end{equation}
It is noteworthy, though, that whilst hyponormality expresses a certain interrelationship between
$A$ and its adjoint, criterion \eqref{A-cond} instead involves $A$ and its square $A^2$. 
In addition it was exemplified in \cite{JJ18logconv} that \eqref{A-cond} is unfulfilled for 
certain explicitly given $A\in \B(H)$, even for some symmetric $n\times n$-matrices, $n\ge 2$.

Moreover, the mixed Dirichlet--Neumann  and 
Dirichlet--Robin realisations $A^+_{\op{DN}}$ and $A^-_{\op{DR}}$, respectively, are variational and elliptic, so they
generate holomorphic semigroups in $L^2(\alpha,\beta)$. But none of them are hyponormal,
cf.\  Example~\ref{AD-exmp} below. This
delicate situation around the $A^{\pm}$ should motivate the present analysis of the generators
that have log-convex decay.
It is envisaged that \eqref{A-cond} can give 
interesting examples when $A$ is a suitable realisation of a partial differential operator.

\bigskip

In the above discussion of log-convexity of $h(t)$, its importance for the
dynamics of \eqref{ivp-id} was explained in (i)--(vi) via the more general strict convexity. So it is natural to pose
the question: does log-convexity have advantages in itself?
 At least it gives rise to the (perhaps new) proof technique used in the next section.

\section{A new necessary condition for log-convex decay}

The reader is assumed familiar with semigroup theory, for which 
\cite{EnNa00, Paz83} could be references; the simpler Hilbert space case is exposed e.g.\ in
\cite[Ch.\ 14]{G09}. 

It is recalled that there is a bijection between the $C_0$-semigroups $e^{-tA}$ in $\B(H)$
that are uniformly bounded, i.e.\  
$\|e^{-tA}\|\le M$ for $t\ge0$, and holomorphic in $\Sigma_\delta\subset \C$ for 
$\delta\in\,]0,\frac\pi2[\,$, and the
densely defined, closed operators $A$ in $H$ satisfying a resolvent estimate 
$|\lambda|\big\|(A+\lambda I)^{-1}\big\|\le C$ for all $\lambda \in \{0\}\cup\Sigma_{\delta+\pi/2}$.

It is classical that, since $\sigma(A)\subset\set{z\in\C}{\Re z\ge\varepsilon}$ for some
$\varepsilon>0$, there is a bound $\|e^{-tA}\|\le M_\eta e^{-t\eta}$ for $t\ge0$,
$0<\eta<\varepsilon$. This yields the crude decay estimate 
\begin{equation} \label{Meta-est}
  h(t)\le M_\eta e^{-t\eta}|u_0|.  
\end{equation}
In general the possible $\eta$ are restricted by $0\le \eta < \underline\sigma(A)$ in terms of
the spectral abscissa of $A$,
\begin{equation}
  \underline\sigma(A)=\inf\Re\sigma(A).  
\end{equation}

The below analyses all rely on the recent result that such semigroups consist
of \emph{injections}, which, mentioned for precision, holds without the uniform boundedness:

 \begin{lem}[\cite{JJ19},\cite{JJ19cor}] \label{inj-lem}
 If $-A$ generates a holomorphic semigroup $e^{-zA}$ in $\B(X)$ for some complex Banach space
 $X$, and $e^{-zA}$ is holomorphic in the open sector $\Sigma_{\delta}\subset \C$ given by $|\arg z|<\delta$ for some  
 $\delta>0$, then $e^{-zA}$ is injective on $X$ for each such $z$.
 \end{lem}
 
The injectivity is clearly equivalent to the geometric property that two solutions
$e^{-tA}v$ and $e^{-tA}w$ to the differential equation $u'+Au=0$ cannot have any points of confluence in $X$ for $t>0$ when $v\ne w$. 
One obvious consequence of this is its backward uniqueness: $u(T)=0$ implies $u(t)=0$ for $0\le
t\le T$. 

Lemma~\ref{inj-lem} is also important because it allows a calculation of $h'(t)$, $h''(t)$, 
using differential calculus in Banach spaces as exposed e.g.\ by H{\"o}rmander \cite[Ch.\ 1]{H} or Lang
\cite{Lan}. 
This uses that 
$u(t)=e^{-tA}u_0\ne0$ for all $t>0$ when $u_0\ne0$, cf.\  Lemma~\ref{inj-lem}, whence $h(t)>0$:

As the  
inner product on $H$, despite its sesquilinearity, is differentiable on the induced real vector space $H_\R$
with derivative $\ip{\cdot}{y}+\ip{x}{\cdot}$ at $(x,y)\in H_\R\oplus H_\R$, which applies 
to the composite map between open sets $\R_+\to (H_\R\setminus\{0\})\oplus (H_\R\setminus\{0\})\to \R_+\to \R_+$  given by  $t\mapsto \sqrt{\ip{u(t)}{u(t)}}$,
the Chain Rule for real Banach spaces gives

\begin{align} \label{h'-id}
  h'(t)&= \frac{\ip{u'}{u}+\ip{u}{u'}}{2\sqrt{\ip{u}{u}}} = -\frac{\Re\ip{Au}{u}}{|u|};
\\  \label{h''-id}
  h''(t)&= \frac{\ip{A^2u}{u}+2\ip{Au}{Au}+\ip{u}{A^2u}}{2|u|}-\frac{(\Re\ip{Au}{u})^2}{|u|^3}.
\end{align}
The second line follows from the first, since $u''=(e^{-tA}u_0)''=A^2e^{-tA}u_0=A^2u$.

When $A$ satisfies \eqref{A-cond}, the short-time behaviour at $t=0$ is via the information on
$h'(0)$ in (iv)--(vi) specifically controlled by $\nu(A)$, and not by its spectrum $\sigma(A)$.
Moreover, the proofs in \cite{JJ18logconv} also gave that  
$h'(0)=\inf h'<0$, which when combined with (vi) shows that $A$ is a bit better 
than accretive ($m(A)\ge0$) in the sense that its numerical range is contained in the \emph{open} right half-plane,
$\nu(A)\subset \set{z\in\C}{\Re z>0}$.
It seems useful to call $A$ a \emph{positively accretive} operator, when it has this property
(milder than strict accretivity \cite{Kat95}), 
and it was shown in \cite{JJ18logconv} that \eqref{A-cond} implies this.

But there is a significantly sharper
necessary condition, which is given already now because of the novelty. Its proof exploits the log-convexity directly:
\begin{prop} \label{necessary-prop}
  If the generator $A$ has log-convex height functions $h(t)$ on $[0,\infty[\,$ for every
  $u_0\ne0$ and the one-sided derivative $h'(0)$ exists and fulfils $h'(0)=-\Re\ip{Au_0}{u_0}$ when 
  $u_0\in D(A)$ with $|u_0|=1$, then $A$ is strictly accretive and
  \begin{equation}
    m(A)=\underline\sigma(A)>0.
  \end{equation}
\end{prop}
\begin{proof}
  The log-convexity means that the continuous function $\log h(t)$ is convex on $[0,\infty[\,$, so 
  its graph lies entirely above each of its half-tangents. Applying this at $t=0$ for 
  $u_0\in D(A)$, $|u_0|=1$, and invoking \eqref{Meta-est}, one finds that
  \begin{equation}
     \log h(0)+t\frac{h'(0)}{h(0)}\le \log h(t)\le \log M_\eta -t\eta\qquad\text{ for $t>0$}.
  \end{equation}
  Indeed, $h(t)$ extends to $t<0$ in a $C^1$-fashion
  along its (half-)tangent at $t=0$, after which the Chain Rule applies to $\log h(t)$.
  (Differentiability of $h(t)$ holds for $t>0$ by \eqref{h'-id}, for $t\le 0$ by construction.)

  Now, the above inequalities being valid for all $t>0$, the graphs of the two first order polynomials cannot intersect,
  so their slopes fulfil $h'(0)\le -\eta$ (as $h(0)=|u_0|=1$). 
  Hence $-h'(0)\ge\underline{\sigma}(A)$, as the spectral abscissa is the supremum of the possible $\eta$;
  cf.\ \eqref{Meta-est} ff.
  The assumption on $h'(0)$ in the statement now gives that for any $u_0\in D(A)$ having $|u_0|=1$,
  \begin{equation}
    \Re\ip{Au_0}{u_0}\ge \underline\sigma(A).
  \end{equation}
  This entails the inequality $m(A)\ge\underline\sigma(A)$, hence strict accretivity since 
  $\underline\sigma(A)>0$.

  However, the strict inequality $m(A)>\underline\sigma(A)$ is impossible, for it would imply that $\overline{\nu(A)}$ 
  is contained in the closed half-plane $\Pi_{m(A)}=\set{z}{\Re z\ge m(A)}$ and that 
  $\C\setminus\Pi_{m(A)}=\set{z}{\Re z<m(A)}$
  contains some $\lambda\in\sigma(A)$ as well as $\R_-$ in the resolvent set $\rho(A)$; but then 
  $\sigma(A)$ and $\rho(A)$ intersect the same connectedness component of $\C\setminus\overline{\nu(A)}$, 
  contradicting \cite[Thm.\ 1.3.9]{Paz83}. Hence $m(A)=\underline\sigma(A)$ as claimed.
\end{proof}

\section{Main Results}  \label{disc-sect}
For the reader's sake,
some basics are recalled here: a positive function $f\colon \R\to [0,\infty[\,$ is \emph{log-convex}
if $\log f(t)$ is convex, or more precisely, for all $r\le  t$ in $\R$ and $0<\theta<1$,
\begin{equation} \label{f-logcon}
  f((1-\theta)r+\theta t)\le  f(r)^{1-\theta}f(t)^\theta.
\end{equation}
Note, though, that $t^\theta$ and $t^{1-\theta}$ 
do not require their continuous extensions to $t=0$ when we take $f=h$ below, for since $e^{-tA}$ 
is holomorphic, $h(t)>0$ or equivalently $e^{-tA}u_0\ne0$ holds for $t\ge0$ by Lemma~\ref{inj-lem}.

For the intermediate point $s=(1-\theta)r+\theta t$ an exercise yields $\theta =(s-r)/(t-r)$, so
log-convexity therefore means that, for $0\le r< s<t$,
\begin{equation} \label{f-logcon'}
  f(s) \le f(r)^{1-\frac{s-r}{t-r}} f(t)^{\frac{s-r}{t-r}}.
\end{equation}
This leads to \eqref{h-logcon} for the semigroup.
There  $A$ is just a positive scalar if $\dim H=1$, so \eqref{h-logcon} is then an identity.
For  $\dim H>1$,  the possible validity of \eqref{h-logcon} is by no means obvious to discuss for the operator function
 $e^{-tA}$ in $\B(H)$.

In general log-convexity is stronger than strict convexity for non-constant functions:

\begin{lem} \label{logcon-lem}
  If $f\colon I\to [0,\infty[\,$ is log-convex on an interval or half\-line $I\subset\R$, then
  $f$ is convex---and if $f$ is not constant in any subinterval, then $f$ is \emph{strictly} convex
  on $I$.
\end{lem}
\begin{proof}
  Convexity on $I$ follows from Young's inequality for the dual exponents $1/\theta$ and
  $1/(1-\theta)$:
  \begin{equation}
    f((1-\theta)r+\theta t)\le f(r)^{1-\theta}f(t)^\theta \le (1-\theta)f(r)+\theta f(t).
  \end{equation}

In case $f(r)\ne f(t)$, the last inequality will be strict, as equality holds in Young's inequality if
and only if the numerators are identical (cf.\ \cite[p.\ 14]{NiPe06}). 
This yields strict convexity in this case. 

If there is a common value $C=f(r)=f(t)$ for some $r<t$ in $I$, there is by assumption a
$u\in\,]r,t[\,$ so that $f(u)\ne f(r)$, and because of the convexity of $f$ this entails
that $f(u)<f(r)=f(t)$: when $r<s\le u$ one may write $s=(1-\theta)r+\theta u$ and $s=(1-\omega)r+\omega t$ for suitable
$\theta,\omega\in\,]0,1[\,$, so clearly 
\begin{equation}
  \begin{split}
  f(s)&\le(1-\theta)f(r)+\theta f(u)
\\
      &<(1-\theta)f(r)+\theta f(t) = C= (1-\omega)f(r)+\omega f(t);  
  \end{split}  
\end{equation}
similarly for $u\le s<t$; so $f$ is strictly convex.
\end{proof}

As examples it is noted that whilst $e^{t}$ is log-convex, $f(t)=e^{t}-1$ is not log-convex as
$(\log f)''<0$. However, when $f\colon I\to \,]0,\infty[\,$ is log-convex, so is the stretched function
defined for $a<b$ in $I$ as
\begin{equation}
  f_{a,b}(t)=
  \begin{cases}
    f(t)\quad\text{for $t<a$}, \\ f(a)\quad \text{for $a\le t<b$}, \\ f(t-b)\quad\text{for $b\le t$}.
  \end{cases}
\end{equation}
This follows from the geometrically obvious fact that the convexity of $\log f$ survives the
stretching. Since $f_{a,b}$ clearly is not strictly convex, the last assumption of
Lemma~\ref{logcon-lem} is necessary. Moreover, a small exercise yields, cf.\ \cite{JJ18logconv},

\begin{lem} \label{decay-lem}
  If $f\colon [0,\infty[\,\to \R_+$ is convex and $f(t)\to 0$ for
  $t\to\infty$, then $f$ is \emph{strictly} monotone decreasing.
\end{lem}

By now it is obvious that if a height function $h(t)$ is log-convex on $[0,\infty[\,$ for
some $u_0\ne0$, it fulfils the first assumption in Lemma~\ref{decay-lem} by the convexity
statement in Lemma~\ref{logcon-lem}, and the second assumption holds because of \eqref{Meta-est}.
Therefore such $h(t)$ is necessarily strictly decreasing on $[0,\infty[\,$---hence non-constant 
in any subinterval, and by Lemma~\ref{logcon-lem} therefore strictly convex.

That $h(t)>0$ allows an analysis of its log-convexity using a characterisation of the
log-convex $C^2$-functions as the solutions to a differential inequality:

\begin{lem} \label{logconvex-lem}
If $f\in C([0,\infty[\,,\R_+)$ is $C^2$ for $t>0$,
the following are equivalent:
\begin{Rmlist}
  \item $f'(t)^2\le f(t)f''(t)$ holds whenever $0<t<\infty$.
  \item $f(t)$ is log-convex on the open half\-line $\,]0,\infty[\,$, cf.\ \eqref{f-logcon'}.
\end{Rmlist}
In the affirmative case $f(t)$ is log-convex also on the closed half\-line $[0,\infty[\,$.
\end{lem}

\begin{proof}
By the assumptions $F(t)=\log f(t)$ is defined for $t\ge0$ and $C^2$
for $t>0$ and
\begin{equation}
  F''(t)=\Big (\frac{f'(t)}{f(t)}\Big)'=\frac{f''(t)f(t)-f'(t)^2}{f(t)^2}.
\end{equation}
Hence (I) is equivalent to $F''(t)\ge 0$ for $t>0$, which is the criterion for the $C^2$-function 
$F$ to be convex for $t>0$; which is a paraphase of the condition (II) for log-convexity of the 
positive function $f(t)$ for $t>0$.  

Letting $r\to 0^+$ for fixed $s<t$, the continuity of $f(r)$ and 
of, say $\exp({\frac{t-s}{t-r}}\log f(r))$, yields that 
\eqref{f-logcon'} is valid for $0=r<s<t$. So $f$ is log-convex on $[0,\infty[\,$. 
\end{proof}

  The formulation of the lemma was inspired by the discussion of
  convexity notions in \cite{NiPe06}.
   Whilst $f$ in $C^2$ is convex if and only if 
  $f''\ge0$, this positivity is clearly fulfilled if $f$ satisfies
  (I), as $f(t)>0$ is assumed---but the positivity then holds in a qualified way, equivalent to
  log-convexity, since (I)$\iff$(II).

The differential inequality in (I) of Lemma~\ref{logconvex-lem} is straightforwardly seen to amount to the following for 
$h(t)$, cf.\ \eqref{h'-id}--\eqref{h''-id}, 
\begin{equation} \label{h-ineq'}
  2(\Re\ip{Au}{u})^2\le   \big(\Re\ip{A^2u}{u}+|Au|^2\big)|u|^2.
\end{equation}
Obviously this  is fulfilled for every $t>0$ when $A$ satisfies \eqref{A-cond}
above, for $u(t)=e^{-tA}u_0$ belongs to the subspace $D(A^n)\subset D(A^2)$ for every $n\ge2$, 
and all $u_0\in H$, when the semigroup is holomorphic. Moreover, the continuity of $h(t)$ and of
its derivatives $h'$, $h''$ given above show that 
$h\in C^2$ for $t>0$. 
So according to  Lemma~\ref{logconvex-lem}, condition \eqref{A-cond} implies that $h(t)=|e^{-tA}u_0|$ is log-convex
on the closed half-line $[0,\infty[\,$.

Conversely, when the height function $h(t)$ is log-convex for each $u_0\ne0$, then the generator $-A$
fulfils \eqref{A-cond}. Indeed, $h$ then fulfils (I) above by the log-convexity,
hence \eqref{h-ineq'} holds. Especially it is seen by insertion of an arbitrary $u_0\in D(A^2)$ in
\eqref{h-ineq'} and commutation of $A$ and $A^2$ with the semigroup that
\begin{equation} 
  2(\Re\ip{e^{-tA}Au_0}{e^{-tA}u_0})^2\le   \Big(\Re\ip{e^{-tA}A^2u_0}{e^{-tA}u_0}+|e^{-tA}Au_0|^2\Big)|e^{-tA}u_0|^2.
\end{equation}
By passing to the limit for $t\to 0^+$ it follows by continuity that
\eqref{A-cond} holds for $x=u_0$. 

Consequently \eqref{A-cond} characterises the generators $-A$ of
uniformly bounded, analytic semigroups having log-convex height functions for all non-trivial initial data.

The above discussion now allows the following sharpening of \cite[Thm.\ 2.5]{JJ18logconv}:

\begin{thm} \label{logcon-thm}
  When $-A$ denotes a generator of a uniformly bounded, holomorphic $C_0$-semigroup $e^{-tA}$ in a
  complex Hilbert space $H$, then the following properties are equivalent:
  \begin{Rmlist}
    \item 
        $2\big(\Re\ip{Ax}{x}\big)^2\le \Re\ip{A^2x}{x}|x|^2+|Ax|^2|x|^2$ for every $x\in D(A^2)$.
   
\item $h(t)=|e^{-tA}u_0|$ is log-convex for every
     $u_0\ne0$; that is, whenever $0\le r<s<t$,
\begin{equation}  \label{h-logcon''}
  \big|e^{-sA}u_0\big| \le \big|e^{-rA}u_0\big|^{\frac{t-s}{t-r}}  \big|e^{-tA}u_0\big|^{\frac{s-r}{t-r}}.
\end{equation}
  \end{Rmlist}
In the affirmative case, $h(t)$ is for $u_0\ne0$ strictly positive, strictly decreasing  
and strictly convex on the closed half\-line $[0,\infty[\,$ and moreover differentiable from the right at $t=0$, 
with a derivative in $[-\infty,0[\,$, 
which for $|u_0|=1$ satisfies
\begin{align}  \label{h'mA-ineq}
  h'(0)&=\inf_{t>0} h'(t)\le - m(A)<0;
\\
\intertext{and if $u_0\in D(A)$ with $|u_0|=1$, then $h\in C^1([0,\infty[\,,\R)\bigcap C^\infty(\R_+,\R)$ and}
  h'(0)&=-\Re\ip{Au_0}{u_0}.
 \label{h'-eq}
\end{align}
Furthermore $\underline{\sigma}(A)=m(A)>0$ holds, in particular such  $A$ are strictly accretive.
\end{thm}

\begin{proof}
That (I)$\iff$(II) was seen in the considerations after Lemma~\ref{logconvex-lem}. The 
strict positivity was derived after 
Lemma~\ref{inj-lem}, strict decrease and strict convexity after Lemma~\ref{decay-lem}.

Convexity of $h$ entails  $h''(t)\ge0$ for $t>0$, so $h'(t)$ is increasing on $\R_+$ and
 $\lim_{t\to 0^+ }h'(t)=\inf_{t>0} h'$ exists in $[-\infty,0[\,$, as $h'<0$. 
By the Mean Value Theorem, some $t'\in\,]0,t[\,$ fulfils
\begin{equation}  \label{h't'-id}
  (h(t)-h(0))/t=h'(t')<0.
\end{equation}
Therefore  $h(t)$ is (extended) differentiable from the right at $t=0$, with $h'(0)=\inf h'$. 
Since the strong continuity and strict decrease of $h$ gives  $|e^{-tA}u_0|\nearrow 1$ for $t\to 0^+$, 
an application of \eqref{h'-id} yields
\begin{equation}
  h'(0)=\inf h'\le \limsup_{t\to0^+}h'(t)\le \limsup_{t\to0^+} (-m(A)|e^{-tA}u_0|) \le -m(A).
\end{equation}

In case $u_0\in D(A)$ and $|u_0|=1$, one can exploit that $h'(0)=\lim_{t\to 0^+}h'(t)$ by commuting $A$ with 
$e^{-tA}$ in \eqref{h'-id}, which in the limit gives, 
because of the strong continuity at $t=0$ and the continuity of inner products, 
\begin{equation} \label{h'0-id}
  h'(0)=\lim_{t\to 0^+} -\Re\ip{e^{-tA}Au_0}{e^{-tA}u_0} 
         = -\Re\ip{Au_0}{u_0}.
\end{equation}
In addition, it is seen that $h'(0)$ is a real number for
$u_0\in D(A)$, so $h\in C^1([0,\infty[\,,\R)$ for such $u_0$. 
For general $u_0\in H$ it follows from the Chain Rule that $h\in C^\infty(\R_+,\R)$. 

Finally, the last line of the statement results from Proposition~\ref{necessary-prop}.
\end{proof}

The conclusions of the theorem apply in particular to every hyponormal generator $-A$, cf.\ the account in 
\eqref{hypo-ver} that such $A$ always satisfy the criterion \eqref{A-cond}.

It is instructive to review condition \eqref{A-cond} in case the generator $A$ is variational.
That is, for some Hilbert space $V\subset H$ algebraically, topologically and densely and some
sesquilinear form $a\colon V\times V\to \C$, which is $V$-bounded and $V$-elliptic in the sense
that (with $\|\cdot\|$ denoting the norm in $V$) for some $C_0>0$
\begin{equation}
  \Re a(u,u)\ge C_0\|u\|^2 \qquad\text{for all $u\in V$},
\end{equation}
it holds for $A$ that $\ip{Au}{v}=a(u,v)$ for all
$u\in D(A)$ and  $v\in V$. Lax--Milgram's lemma on the properties of $A$ is
exposed in \cite[Ch.\ 12]{G09} and \cite[Ch.\ 3]{Hel13}.
It is classical that $-A$ generates a holomorphic semigroup $e^{-tA}$ in $\B(H)$; an explicit proof
is e.g.\ given in \cite[Lem.\ 4]{ChJo18ax}.

For such $A$, the log-convexity criterion \eqref{A-cond} can be stated for $V$-elliptic variational $A$ as a
comparison of sesquilinear forms, cf.\ \cite{JJ18logconv},
\begin{equation}
  \big(\Re a(u,u)\big)^2\le \Re \big(a_{\Re}(Au,u)\big)\ip{u}{u}\qquad\text{ for $u\in D(A^2)$}.
\end{equation}

\begin{exmp}  \label{AD-exmp}
  To see that variational operators need not be hyponormal, one may take $H=L_2(\alpha,\beta)$, with norm 
  $\|f\|_0=(\int_{\alpha}^{\beta} |f(x)|^2\,dx)^{1/2}$, for reals $\alpha<\beta$ and let 
  $V=\set{v\in H^1(\alpha,\beta)}{u(\alpha)=0}$ be a subspace of the first Sobolev space with norm
  given by $\|f\|_1^2=\int_{\alpha}^{\beta} (|f(x)|^2+|f'(x)|^2)\,dx$ and the  sequilinear forms
  \begin{equation}
    a(u,v)=\int_{\alpha}^{\beta} u'(x)\overline{v'(x)}+ u'(x)\overline{v(x)}\,dx.
  \end{equation}
 This is clearly $V$-bounded, and also $V$-elliptic: partial integration gives 
 $\Re a(u,u)=\|u'\|_0^2+\frac12|u(\beta)|^2$, and
 $\Re a(u,u)\ge C_0\|u\|_1^2$ follows for all $u\in V$ and e.g.\ $C_0=\min(\frac12,(\beta-\alpha)^{-2})$ by ignoring
   the last term and using Poincar\'e's inequality (its standard proof,
   e.g.\ \cite[Thm.\ 4.29]{G09}, applies to $V$).

 The induced $A^+_{\op{DN}}$ acts in the distribution space $\mathcal{D}'(\alpha,\beta)$ of Schwartz \cite{Swz66} as 
 $A^+_{\op{DN}}u=-u''+u'$,
 which is the advection-diffusion operator with a mixed Dirichlet and Neumann condition,
 \begin{equation}
   D(A^+_{\op{DN}})=\Set{u\in H^2(\alpha,\beta)}{u(\alpha)=0,\ u'(\beta)=0}.
 \end{equation}
(The Dirichlet realisation of $u'-u''$ has been studied at length; cf.\
\cite[Ch.\ 12]{EmTr05}.)

 As $(A^+_{\op{DN}})^*$ is induced by $\overline{a(v,u)}$, one finds similarly 
 $(A^+_{\op{DN}})^*u=-u''-u'=A^-_{\op{DR}}u$ with the domain characterised by a mixed Dirichlet and Robin condition,
 \begin{equation}
   D((A^+_{\op{DN}})^*)=D(A^-_{\op{DR}})=\Set{u\in H^2(\alpha,\beta)}{u(\alpha)=0,\ u'(\beta)+u(\beta)=0}.
 \end{equation}
 As both $D(A^+_{\op{DN}})$ and $D((A^+_{\op{DN}})^*)$ contain functions outside their intersection, 
 \eqref{hypo-normal} shows
 that neither $A^+_{\op{DN}}$ nor $(A^+_{\op{DN}})^*=A^-_{\op{DR}}$ is hyponormal. This is part of the motivation for the study of
 condition \eqref{A-cond}.
\end{exmp}


\providecommand{\bysame}{\leavevmode\hbox to3em{\hrulefill}\thinspace}
\providecommand{\MR}{\relax\ifhmode\unskip\space\fi MR }
\providecommand{\MRhref}[2]{%
  \href{http://www.ams.org/mathscinet-getitem?mr=#1}{#2}
}
\providecommand{\href}[2]{#2}

\end{document}